\documentclass[12pt,reqno]{amsart}

\usepackage{amssymb}
\usepackage{amscd}
\usepackage{amsfonts}
\usepackage{setspace}
\usepackage{version}

\usepackage{graphicx}

\newtheorem{theorem}{Theorem}[section]
\newtheorem{lemma}[theorem]{Lemma}

\theoremstyle{definition}

\renewcommand{\leq}{\leqslant}
\renewcommand{\geq}{\geqslant}

\def\F{\mathbb{F}}

\def\Z{\mathbf{Z}}

\def\P{\mathbb{P}}

\def\eps{\varepsilon}

\title{An example concerning set addition in $\F_2^n$}
\author{Ben Green}
\email{ben.green@maths.ox.ac.uk}
\address{Mathematical Institute, University of Oxford}
\author{Daniel Kane}
\email{dakane@ucsd.edu}
\address{Department of Mathematics, University of California at San Diego}

\numberwithin{equation}{section}

\begin{document}
\maketitle

\begin{abstract}
We construct sets $A, B$ in a vector space over $\F_2$ with the property that $A$ is ``statistically'' almost closed under addition by $B$ in the sense that $a + b$ almost always lies in $A$ when $a \in A, b \in B$, but which is extremely far from being ``combinatorially'' almost closed under addition by $B$: if $A' \subset A$, $B' \subset B$ and $A' + B'$ is comparable in size to $A'$ then $|B'| \lessapprox |B|^{1/2}$.
\end{abstract}

\section{Introduction}
The aim of this note is to prove the following result, which is a precise version of the claim made in the abstract.

\begin{theorem}\label{mainthm}
Let $0 < \delta < 1 < K$. Then there are arbitrarily large finite sets $A, B$ in some vector space over $\F_2$ with the property that $\P(a + b \in A | a \in A, b \in B) \geq 1 - \delta$, but such that if $A' \subset A$ and $B' \subset B$ satisfy $|A' + B'| \leq K|A'|$ then $|B'| = O_{K, \delta}(|B|^{1/2})$.
\end{theorem}

The sets are constructed as follows. For positive integer $m$ we take $A$ to consist of a direct product of  $m$ ``bands'' (union of a few consecutive Hamming layers) in copies of $\F_2^m$. Then we take $B$ to consist of the standard basis vectors in $(\F_2^m)^m$. The details may be found in Section \ref{sec3}.

Whilst the construction is fairly simple, neither the construction nor its analysis seemed quite obvious to us and so we thought it worth recording in this short note.

\section{A crude isoperimetric inequality for bands}

By a \emph{band} of width $k$ in $\F_2^D$ we mean the union $L_{\ell} \cup \dots \cup L_{\ell + k - 1}$ of $k$ consecutive layers, where we define the layer $L_d := \{x \in \F_2^D : |x| = d\}$, where $|x|$ is the number of nonzero coordinates of $x$.

\begin{lemma}\label{lem2.1}
Suppose that $A \subset \F_2^D$ is contained in some band $\Sigma$ of width $k$, where $D > 10k^2$. Let $S = \{e_1,\dots, e_D\}$ be the set of standard basis vectors in $\F_2^D$. Then 
\[ |(A + kS) \setminus \Sigma| \geq \frac{1}{2}|A|.\]
\end{lemma}
\begin{proof}
We begin by noting that if $x \in L_i$ then there are $\binom{D - i}{k}$ elements $y \in L_{i + k}$ differing from $x$ by an element of $kS$. Each such element arises from $\binom{i+k}{k}$ choices of $x$. It follows by double counting that if $A_i \subset L_i$ is some set then
\[ |(A_i + kS) \cap L_{i + k}| \geq \frac{\binom{D - i}{k}}{\binom{i + k}{k}}  |A_i|.\]
If $i \leq D/2$ then we have
\[ \frac{\binom{D - i}{k}}{\binom{i + k}{k}} \geq \big(\frac{\frac{D}{2} - k}{\frac{D}{2} + k}   \big)^k \geq (1 - \frac{4k}{D})^k > \frac{1}{2}.\] Thus if $i \leq D/2$ then
\[ |(A_i + kS) \cap L_{i + k}| \geq \frac{1}{2}|A_i|.\]

Suppose that $\Sigma = L_{\ell} \cup \dots \cup L_{\ell + k - 1}$, where $\ell \leq D/2$. Then, applying the preceding inequality with $i = \ell, \dots, \ell + k - 1$ in turn and with $A_i := A \cap L_i$, the claimed result follows by noting that all of the layers $L_{i + k}$ lie outside $\Sigma$.

If $\ell > D/2$ then a very similar argument works, but we instead use the inequality
\[ |(A_i + kS) \cap L_{i - k}| \geq \frac{\binom{i}{k}}{\binom{D + k - i}{k}}|A_i| \geq \frac{1}{2}|A_i|\] for $i \geq D/2$.

\end{proof}

\section{The example}\label{sec3}

Let $0 < \delta < 1 < K$, as in the statement of Theorem \ref{mainthm}. Let $k = 2k' + 1$ be odd and of size $\sim \frac{4}{\delta}$, and let $m = 2m'$ be even with $m > 10k^2$. Consider the vector space $V = (\F_2^m)^m$, and let $B \subset V$ be the standard basis consisting of elements $(e_{i,j})_{i,j \in [m]}$.  Let $A \subset V$ consist of the direct product of $m$ copies of the band $\Sigma := L_{m'- k'} \cup \cdots \cup L_{m' + k'}$ in $\F_2^m$. 

It is straightforward to show that $A$ is almost closed under addition by $B$ in the ``statistical'' sense. 

\begin{lemma}\label{lem3.1}
We have $\P(a + b \in A | a \in A, b \in B) \geq 1 - \frac{2}{k} > 1 - \delta$. 
That is, 
\[ \big| \{ a \in A, b \in B : a + b \in A\}\big|   > (1 - \delta) |A| |B|.\]
\end{lemma}
\begin{proof}
By symmetry it suffices to show that $\P(a + e_{1,1} \in A | a \in A) \geq 1 - \frac{2}{k}$. If the projection of $a$ to the first copy of $\F_2^m$ lies in $L_{m' - k' + 1} \cup \cdots \cup L_{m' + k' - 1}$, then it is evidently the case that $a + e_{1,1} \in A$. The probability of this event is precisely
\[ \frac{ |L_{m' - k' + 1} \cup \cdots \cup L_{m' + k' - 1}|}{|L_{m' - k'} \cup \cdots \cup L_{m' + k'}|} \geq \frac{2k' - 1}{2k' + 1} = 1 - \frac{2}{k},\] since the sizes of the layers decrease away from the middle layer.
\end{proof}

Shortly we will turn to the somewhat less immediate task of showing that $|A' + B'|$ is much bigger than $|A'|$ unless $|B'|$ is of size $\ll m = |B|^{1/2}$. We note that there are two quite distinct examples of sets $B'$ of size $m$ for which there is some $A' \subset A$ almost closed under addition by $B'$. Indeed
\begin{itemize}
\item If $B' = \{e_{i,1}\}_{i \in [m]}$ then $A + B'$ consists of the product of the band of width $(k+2)$, $L_{m' - k' - 1} \cup \cdots \cup L_{m' + k' + 1}$, with $(m-1)$ copies of the usual band of width $k$. By a computation very similar to that in the proof of Lemma \ref{lem3.1}, we have $|A + B'| \leq (1 + O(\frac{1}{k})) |A|$.
\item If $B' = \{e_{1,j}\}_{j \in [m]}$ then we may take $A' \subset A$ to be the product of some $e_{1,1}$-invariant set in the first factor with an $e_{1,2}$-invariant set in the second factor and so on to obtain $A' + B' = A'$. Note that any band in $\F_2^m$ of width $\geq 2$ does contain a set invariant under a given basis vector $e_{i,j}$.
\end{itemize}
These examples, in addition to showing that we cannot improve upon our bound of $|B'| \lessapprox |B|^{1/2}$, should also convince the reader that there is nothing to be gained by taking our set $A$ to live inside $(\F_2^m)^r$ with differing values of $m,r$: if $m > r$ then examples of the first type are worse, whilst if $m < r$ then examples of the second type are bad.\vspace{11pt}

Suppose that 
\begin{equation}\label{supposition}  |A' + B'| \leq K|A'|\end{equation} for some $A' \subset A$ and $B' \subset B$. 
We begin by applying a variant of the Pl\"unnecke-Ruzsa inequality due to Ruzsa \cite[Chapter 1, Proposition 7.3]{ruzsa}. A special case of this theorem tells us that if $|A' + B'| \leq K|A'|$ then for any positive integer $k$ there is $A'' \subset A'$, 
\begin{equation}\label{add-lower} |A''| \geq \frac{1}{2}|A'|,\end{equation} such that 
\begin{equation}\label{adashdash} |A'' + kB'| \leq 2^{k+1}K^k |A''|.\end{equation}
Now since $B' \subset B$ we have
\[ B' = \bigcup_{j = 1}^m B'_j,\] where
\[ B'_j := \bigcup_{i \in I_j} \{ e_{i,j} \}\]  for some set $I_j \subset [m]$.
Say that $j$ is \emph{good} if $|I_j| > 10 k^2$, and write $J \subset [m]$ for the set of good $J$. We claim that if $j$ is good then
\begin{equation}\label{claim-grow} |(A'' + kB'_j) \setminus A| \geq \frac{1}{2}|A''|.\end{equation}
The sets $(A'' + k B'_j) \setminus A$ are disjoint for different $j$, as all the elements of such a set, considered as a subset of $(\F_2^m)^m$, have their $j$th coordinate, and no other, lying outside the band $\Sigma$. It follows from this and \eqref{adashdash} that 
\[ J \leq 2^{k+2} K^k,\] and therefore we have
\[ |B'| = \sum_{j = 1}^m |I_j| \leq 10k^2 m + J m \ll_{K,\delta} m.\] 
It remains to establish the claim \eqref{claim-grow}. To do this, identify $B'_j$ with the standard basis vectors in $\F_2^{I_j}$, and write $V$ as a direct product $\F_2^{I_j} \times V'$. Then $A,A''$ may be fibred as unions
\[ A = \bigcup_{v \in V'} (A(v) \times \{v\})\]
\[ A'' = \bigcup_{v \in V'} (A''(v) \times \{v\})\] 
where $A''(v) \subset A(v) \subset \F_2^{I_j}$ and, crucially, each $A(v)$ is a band of width $k$ in $\F_2^{I_j}$. By Lemma \ref{lem2.1} (with $D = |I_j| > 10k^2$) it follows that 
\[ |(A''(v) + kB'_j) \setminus A(v)| \geq \frac{1}{2}|A''(v)|,\] from which the claim follows by summing over $v$.

\section{Acknowledgements} The first author thanks Shachar Lovett for asking a question which led to this note. Lovett (personal communication) subsequently informed us that Hosseini, Impagliazzo and he observed that easier examples are available in other ambient spaces. For example, $A = \{0,1\}^n$, $B = \{e_1,\dots, e_n\}$ (viewed as subsets of $\Z^n$, or of $\F_p^n$ for any $p \geq 3$) have similar properties. In fact for this example we have $\P(a + b \in A| a \in A, b \in B) = \frac{1}{2}$ and $|A' + B'| \leq K|A'|$ implies that $|B'| \ll_{K} 1$. A somewhat similar example is mentioned in \cite[Exercise 2.6.2]{tao-vu}.

The authors thank MSRI and the Simons Institute for providing excellent working conditions while this work was carried out. The first author is a Simons Investigator. The second author is supported by NSF CAREER award number 1553288.

\end{document}